\documentclass[a4paper,12pt,oneside,reqno]{amsart}

\usepackage[utf8]{inputenc}
\usepackage{amsfonts}
\usepackage{amsmath}
\usepackage{amssymb}
\usepackage{amsthm}
\usepackage{enumerate}
\usepackage{geometry}

\newcommand{\1}{1\!\!\,\mathrm{I}}
\newcommand{\diam}{\mathrm{diam}\,}
\newcommand{\supp}{\mathrm{supp}\;}

\theoremstyle{plain}
\newtheorem{theorem}{Theorem}[section]
\newtheorem{lemma}[theorem]{Lemma}
\newtheorem{corollary}[theorem]{Corollary}

\theoremstyle{definition}
\newtheorem{definition}[theorem]{Definition}

\numberwithin{equation}{section}
\title
[The rate of weak convergence of the $n$-point motions of Harris flows]
{The rate of weak convergence of\\ the $n$-point motions of Harris flows}
\author{A.~A.~Dorogovtsev}
\author{V.~V.~Fomichov}
\address{Institute of Mathematics, National Academy of Sciences of Ukraine, Tereshchen\-kivska str.~3, Kiev~01004, Ukraine}
\email{adoro@imath.kiev.ua}
\email{v-vfom@imath.kiev.ua}
\keywords{Harris flows, $n$-point motions, random measures, weak convergence}
\subjclass[2010]{60G57, 60G60, 60K35.}

\begin{document}

\begin{abstract}
In this paper we study the Wasserstein distance between the distributions of the $n$-point motions of one-dimensional Harris flows whose covariance functions have compact support. We prove that it can be estimated by the diameters of the support of the covariance functions provided the latter are sufficiently small.
\end{abstract}

\maketitle

\section{INTRODUCTION}

In this paper we study the Wasserstein distance between the distributions of the $n$-point motions of one-dimensional Harris flows whose covariance functions have compact support. For convenience let us recall the definition of a Harris flow.

\begin{definition}
A random field $\{x(u,t),\; u\in\mathbb{R},\; t\geqslant 0\}$ is called {\it a Brownian stochastic flow} if it satisfies the following conditions:
\begin{enumerate}
\item[1)]
for any $u\in\mathbb{R}$ the stochastic process $\{x(u,t),\; t\geqslant 0\}$ is a Brownian motion with respect to the common filtration $(\mathcal{F}_t:=\sigma\{x(v,s),\; v\in\mathbb{R},\; 0\leqslant s\leqslant t\})_{t\geqslant 0}$ such that $x(u,0)=u$;

\item[2)]
for any $u,v\in\mathbb{R}$, if $u\leqslant v$, then $x(u,t)\leqslant x(v,t)$ for all $t\geqslant 0$.
\end{enumerate}
\end{definition}

\begin{definition}
A Brownian stochastic flow $\{x(u,t),\; u\in\mathbb{R},\; t\geqslant 0\}$ is called {\it a Harris flow} with covariance function $\Gamma$ if for any $u,v\in\mathbb{R}$ the joint quadratic variation of the martingales $\{x(u,t),\; t\geqslant 0\}$ and $\{x(v,t),\; t\geqslant 0\}$ is given by
$$
\left\langle x(u,\cdot),x(v,\cdot)\right\rangle_t=\int\limits_0^t \Gamma(x(u,s)-x(v,s))\,ds,\quad t\geqslant 0.
$$
\end{definition}

Note that the function $\Gamma$ is necessarily non-negatively definite, symmetric, and
$$
\Gamma(0)=1.
$$

The historically first example of a Brownian stochastic flow was constructed by R.~A.~Arratia in~\cite{Arratia} as a weak limit of families of coalescing simple random walks. For the Arratia flow $\{x_0(u,t),\; u\in\mathbb{R},\; t\geqslant 0\}$ one has
$$
\forall u,v\in\mathbb{R}:\quad \left\langle x_0(u,\cdot),x_0(v,\cdot)\right\rangle_t= \int\limits_0^t \1_{\{0\}}(x_0(u,s)-x_0(v,s))\,ds,\quad t\geqslant 0,
$$
where $\1_{\{0\}}$ stands for the indicator function of the set $\{0\}$, and so,
$$
\Gamma=\1_{\{0\}}.
$$
Informally one can describe the Arratia flow as a flow of Brownian particles in which any two particles move independently until they meet and after that coalesce and move together.

Later, in~\cite{Harris} T.~E.~Harris proved the existence of a generalisation (in some sense) of the Arratia flow for covariance functions $\Gamma$ which are continuous on $\mathbb{R}$ and satisfy the Lipschitz condition on all sets of the form $\mathbb{R}\backslash (-\delta;\delta)$, $\delta>0$.

In the case when $\Gamma$ is smooth enough the corresponding Harris flow can be obtained as the flow of solutions of a stochastic differential equation. To be more precise, let us take a function $\varphi\in C_0^2(\mathbb{R})$ (i.~e. $\varphi$ belongs to $C^2(\mathbb{R})$ and has compact support) such that
$$
\int\limits_\mathbb{R} \varphi^2(q)\,dq=1,
$$
and for $u\in\mathbb{R}$ consider the following Cauchy problem:
\begin{equation*}
\begin{cases}
dx(u,t)=\int\limits_\mathbb{R} \varphi(x(u,t)-q)\,W(dq,dt),\quad t\geqslant 0,\\
x(u,0)=u,\\
\end{cases}
\end{equation*}
where $W$ is a Wiener sheet on $\mathbb{R}\times [0;+\infty)$ (on integration with respect to a Wiener sheet see~\cite{Dorogovtsev2007}, \cite{Kotelenez}, \cite{Walsh}). The conditions on the function $\varphi$ imply that for every $u\in\mathbb{R}$ this Cauchy problem has a unique (strong) solution $\{x(u,t),\; t\geqslant 0\}$. It is easy to check that the random field $\{x(u,t),\; u\in\mathbb{R},\; t\geqslant 0\}$ is a Harris flow with covariance function $\Gamma$ given by
$$
\Gamma(z)=\int\limits_\mathbb{R} \varphi(z-q)\varphi(-q)\,dq\equiv\int\limits_\mathbb{R} \varphi(z+q)\varphi(q)\,dq,\quad z\in\mathbb{R}.
$$
Indeed, from the properties of the integral with respect to a Wiener sheet it follows that for any $u,v\in\mathbb{R}$ the joint quadratic variation of the continuous square-integrable martingales $\{x(u,t),\; t\geqslant 0\}$ and $\{x(v,t),\; t\geqslant 0\}$ is given by
$$
\left\langle x(u,\cdot),x(v,\cdot)\right\rangle_t=\int\limits_0^t \Gamma(x(u,s)-x(v,s))\,ds,\quad t\geqslant 0.
$$
In particular, for any $u\in\mathbb{R}$ we have
$$
\left\langle x(u,\cdot)\right\rangle_t=t,\quad t\geqslant 0,
$$
and hence, by L\'{e}vy's characterising theorem~\cite[Theorem~3.3.16]{KaratzasShreve}, the stochastic process $\{x(u,t),\; t\geqslant 0\}$ is a Brownian motion. Finally, it remains to note that the condition $\varphi\in C_0^2(\mathbb{R})$ implies that the random mappings
$$
x(\cdot,t)\colon\mathbb{R}\rightarrow\mathbb{R},\quad t\geqslant 0,
$$
are diffeomorphisms (see~\cite{Kunita}), and so, if $u\leqslant v$, then $x(u,t)\leqslant x(v,t)$ for all $t\geqslant 0$.

Let us note that
$$
\Gamma(z)=0,\quad |z|>\dfrac{1}{2}d(\Gamma),
$$
where
$$
d(\Gamma):=\diam(\supp\Gamma),
$$
and hence
$$
\left\langle x(u,\cdot),x(v,\cdot)\right\rangle_{t\wedge\tau}=\int\limits_0^{t\wedge\tau} \Gamma(x(u,s)-x(v,s))\,ds=0,\quad t\geqslant 0,
$$
where
$$
\tau:=\inf\{t\geqslant 0\; \vert\; |x(u,t)-x(v,t)|\leqslant \dfrac{1}{2}d(\Gamma)\}.
$$
So, informally one can say that any two particles of this Harris flow move independently until the distance between them does not reach $\frac{1}{2}d(\Gamma)$. Thus, when $d(\Gamma)$ is close to zero its $n$-point motions are similar to those of the Arratia flow. Moreover, it was proved in~\cite{Dorogovtsev2004} that when $d(\varphi):=\diam(\supp\varphi)$ (or, equivalently, $d(\Gamma)$) tends to zero they converge weakly to the $n$-point motions of the Arratia flow. Our aim in this paper is to estimate the rate of this convergence.

To formulate our main result we need some notations. They will be used throughout the rest of the paper.

For a complete separable metric space $(X,d)$ let $\mathcal{P}(X)$ denote the set of all Borel probability measures on $X$ and define
$$
\mathcal{M}_1(X):=\{\mu\in\mathcal{P}(X)\, |\, \int\limits_X d(u,u_0)\,\mu(du)<+\infty\},
$$
where $u_0$ is a fixed point in $X$. It can be easily checked that the set $\mathcal{M}_1(X)$ does not depend on the choice of this point. On $\mathcal{M}_1(X)$ we will consider the standard Wasserstein metric $W_1$ defined by
$$
W_1(\mu',\mu''):=\inf_{\varkappa\in C(\mu',\mu'')} \iint\limits_{X^2} d(u,v)\,\varkappa(du,dv),\quad \mu',\mu''\in\mathcal{M}_1(X),
$$
where $C(\mu',\mu'')$ is the set of all Borel probability measures on $X^2\equiv X\times X$ with marginals $\mu'$ and $\mu''$. It is well known that $(\mathcal{M}_1(X),W_1)$ is also a complete separable metric space (see, for instance,~\cite[Theorem~6.18]{Villani}).

For a Brownian stochastic flow $\{x(u,t),\; u\in\mathbb{R},\; t\geqslant 0\}$ and a measure $\mu\in\mathcal{P}(\mathbb{R})$ set
$$
\lambda:=\mu\circ x^{-1}(\cdot,1),
$$
where $x^{-1}(\cdot,1)$ stands for the ('omegawise', i.~e. for every fixed $\omega\in\Omega$) inverse of the mapping $x(\cdot,1)\colon\mathbb{R}\rightarrow\mathbb{R}$. It can be easily shown that if $\mu\in\mathcal{M}_1(\mathbb{R})$, then $\lambda$ is a random element in $\mathcal{M}_1(\mathbb{R})$. So, we can consider its distribution $\Lambda$ in this space. Note that $\Lambda$ is an element of $\mathcal{M}_1(\mathcal{M}_1(\mathbb{R}))$. With some abuse of notation we will use $W_1$ to denote the Wasserstein distance in both spaces $\mathcal{M}_1(\mathbb{R})$ and $\mathcal{M}_1(\mathcal{M}_1(\mathbb{R}))$.

To avoid defining the corresponding measures each time we need them, we will use the following rule: if not stated otherwise, measures $\lambda$ with an upper and/or lower index will always be defined as above with $\mu$ having the same upper index and/or $x$ having the same lower index, and measures $\Lambda$ with these indices will always denote their distributions in the space $\mathcal{M}_1(\mathbb{R})$.

The main result of this paper is the following theorem.

\begin{theorem}
\label{theorem1}
Let $\{x(u,t),\; u\in\mathbb{R},\; t\geqslant 0\}$ be a Harris flow with covariance function $\Gamma$, which has compact support, and $\{x_0(u,t),\; u\in\mathbb{R},\; t\geqslant 0\}$ be the Arratia flow. Assume that
$$
\supp\mu\subset [0;1]
$$
and
$$
d(\Gamma)<\frac{1}{100}.
$$
Then
$$
W_1(\Lambda,\Lambda_0)\leqslant C \cdot d(\Gamma)^{1/22},
$$
where the constant $C>0$ does not depend on $\mu$ and $\Gamma$.
\end{theorem}

Using the triangle's inequality one obtains the following corollary.

\begin{corollary}
Let $\{x_1(u,t),\; u\in\mathbb{R},\; t\geqslant 0\}$ and $\{x_2(u,t),\; u\in\mathbb{R},\; t\geqslant 0\}$ be two Harris flows with covariance functions $\Gamma_1$ and $\Gamma_2$ respectively, which have compact support. Assume that
$$
\supp\mu\subset [0;1]
$$
and
$$
\max\{d(\Gamma_1),d(\Gamma_2)\}<\frac{1}{100}.
$$
Then
$$
W_1(\Lambda_1,\Lambda_2)\leqslant 2C \cdot \max\{d(\Gamma_1),d(\Gamma_2)\}^{1/22},
$$
where $C>0$ is the constant from Theorem~\ref{theorem1}.
\end{corollary}

To prove Theorem~\ref{theorem1} we approximate the initial measure $\mu$ by discrete measures $\mu^n$ and divide the proof into three steps. In the first step we estimate the Wasserstein distance between $\Lambda$ and $\Lambda^n$ for an arbitrary Brownian stochastic flow. In the second step we use some recursive procedure to construct a suitable coupling of $\lambda^n$ and $\lambda_0^n$ allowing to estimate the Wasserstein distance between their distributions $\Lambda^n$ and $\Lambda_0^n$. In the third step we combine these results and, optimising with respect to $n$, arrive at the desired assertion.

\section{PROOF OF THE MAIN RESULT: FIRST STEP}

Let measure $\mu\in\mathcal{P}(\mathbb{R})$ be such that $\supp\mu\subset [0;1]$. Then, obviously, $\mu$ belongs to $\mathcal{M}_1(\mathbb{R})$ and it can be approximated by a sequence $\{\mu^n\}_{n=1}^\infty\subset\mathcal{M}_1(\mathbb{R})$ of discrete measures defined by
$$
\mu^n:=\sum_{k=1}^n p_k^n\delta_{\frac{2k-1}{2n}},\quad n\geqslant 1,
$$
where
$$
p_k^n:=\mu\left(I_k^n\right),\quad 1\leqslant k\leqslant n,\quad n\geqslant 1,
$$
with
\begin{gather*}
I_k^n:=\left[\frac{k-1}{n};\frac{k}{n}\right),\quad 1\leqslant k\leqslant n-1,\quad n\geqslant 2,\\
I_n^n:=\left[\frac{n-1}{n};1\right],\quad n\geqslant 1.
\end{gather*}

\begin{theorem}
\label{theorem2}
Let $\{x(u,t),\; u\in\mathbb{R},\; t\geqslant 0\}$ be an arbitrary Brownian stochastic flow. Then
$$
W_1(\Lambda,\Lambda^n)\leqslant\dfrac{K}{\sqrt{n}},
$$
where $K=\sqrt{\frac{64}{3\sqrt{2\pi}}+\frac{1}{4}}$.
\end{theorem}

For the proof of this theorem we use the following lemma proved in \cite{Dorogovtsev2010} (there it was formulated for the case when $t=1$, but the proof, mutatis mutandis, is valid for all $t>0$).

\begin{lemma}\cite[Lemma~5]{Dorogovtsev2010}
\label{lemma1}
Let $\{x(u,t),\; u\in\mathbb{R},\; t\geqslant 0\}$ be an arbitrary Brownian stochastic flow. Then
$$
{\bf E} (x(u,t)-x(v,t))^2\leqslant C_t \cdot |u-v|+|u-v|^2,\quad u,v\in\mathbb{R},\quad t\geqslant 0,
$$
where $C_t=\frac{128t^{3/2}}{3\sqrt{2\pi}}$.
\end{lemma}

\begin{proof}[Proof of Theorem~\ref{theorem2}]
By the definition of the Wasserstein distance $W_1$ we have
$$
W_1(\Lambda,\Lambda^n)=\inf_{\varkappa\in C(\Lambda,\Lambda^n)} \iint\limits_{\mathcal{M}_1^2(\mathbb{R})} W_1(\mu',\mu'')\,\varkappa(d\mu',d\mu'')\leqslant {\bf E} W_1(\lambda,\lambda^n),
$$
where for convenience we set
$$
\mathcal{M}_1^2(\mathbb{R}):=\mathcal{M}_1(\mathbb{R})\times \mathcal{M}_1(\mathbb{R}).
$$
However,
\begin{gather*}
{\bf E} W_1(\lambda,\lambda^n)={\bf E} \inf_{\varkappa\in C(\lambda,\lambda^n)} \iint\limits_{\mathbb{R}^2} |u-v|\,\varkappa(du,dv)\leqslant\\
\leqslant {\bf E} \sum_{k=1}^n \int\limits_{I_k^n} \left|x(u,1)-x\left(\frac{2k-1}{2n},1\right)\right|\,\mu(du)\leqslant\\
\leqslant\sum_{k=1}^n \int\limits_{I_k^n} \sqrt{{\bf E} \left|x(u,1)-x\left(\frac{2k-1}{2n},1\right)\right|^2}\,\mu(du).
\end{gather*}
Thus, using Lemma~\ref{lemma1} we obtain that
\begin{gather*}
{\bf E} W_1(\lambda,\lambda^n)\leqslant \sum_{k=1}^n \int\limits_{I_k^n} \sqrt{C_1 \cdot \left|u-\dfrac{2k-1}{2n}\right|+\left|u-\dfrac{2k-1}{2n}\right|^2}\,\mu(du)\leqslant\\
\leqslant \sum_{k=1}^n p_k^n \cdot \sqrt{C_1 \cdot \dfrac{1}{2n}+\dfrac{1}{4n^2}}\leqslant \dfrac{K}{\sqrt{n}},
\end{gather*}
where $K:=\sqrt{\frac{C_1}{2}+\frac{1}{4}}$. The theorem is proved.
\end{proof}

\section{PROOF OF THE MAIN RESULT: SECOND STEP}

Let $\{x(u,t),\; u\in\mathbb{R},\; t\geqslant 0\}$ be a Harris flow with covariance function $\Gamma$, which has compact support. Fix some $\varepsilon>0$ such that $\varepsilon\geqslant \frac{1}{2}d(\Gamma)$ and arbitrary initial points $u_1<u_2<\ldots<u_n$, $n\geqslant 2$, such that the distance between any two of them is strictly greater than $\varepsilon$.

Set
$$
(z_1(u_1,t),\ldots,z_1(u_n,t)):=(x(u_1,t),\ldots,x(u_n,t)),\quad t\geqslant 0,
$$
and associate with this stochastic process a family $\{\Pi_1(t),\; t\geqslant 0\}$ of random partitions of the set $\{1,2,\ldots,n\}$ defined by the following condition: indices $i$ and $i+1$ belong to the same element of the partition $\Pi_1(t)$ if and only if $|z_1(u_i,t)-z_1(u_{i+1},t)|\leqslant \varepsilon$. Obviously, $\Pi_1(0)=\{\{1\},\{2\},\ldots,\{n\}\}$. Also, let $\sigma_1$ be the first time $t>0$ when the partition $\Pi_1(t)$ changes.

Now for all $k\in\{1,\ldots,n\}$ set
$$
z_2(u_k,t):=
\begin{cases}
z_1(u_k,t),\quad 0\leqslant t<\sigma_1,\\
z_1(u_j,t)+(k-j) \cdot \varepsilon,\quad t\geqslant \sigma_1,
\end{cases}
$$
where $j$ is the least index in the element of $\Pi_1(\sigma_1)$ to which $k$ belongs (if $\sigma_1$ is infinite, the lower expression is just omitted). Similarly, with the stochastic process $\{(z_2(u_1,t),\ldots,z_2(u_n,t)),\; t\geqslant 0\}$ we associate the corresponding family $\{\Pi_2(t),\; t\geqslant 0\}$ of random partitions of the set $\{1,2,\ldots,n\}$ and the random time $\sigma_2$ which is equal to the first time $t>\sigma_1$ when the partition $\Pi_2(t)$ changes (if $\sigma_1$ is infinite, $\sigma_2$ is also set to be infinite).

Continuing in this way we can construct at most $n$ distinct $n$-dimensional stochastic processes.

To study the stochastic processes $\{(z_i(u_1,t),\ldots,z_i(u_n,t)),\; t\geqslant 0\}$, $1\leqslant i\leqslant n$, we need to describe their construction more formally.

Fix $\varepsilon>0$ such that $\varepsilon\geqslant \frac{1}{2}d(\Gamma)$ and let $u_1,u_2,\ldots,u_n\in\mathbb{R}$, $n\geqslant 2$, be such that
\begin{gather*}
u_1<u_2<\ldots<u_n,\\
u_{k+1}-u_k>\varepsilon,\quad 1\leqslant k\leqslant n-1.
\end{gather*}
We define recursively
\begin{gather*}
z_1(u_k,t):=x(u_k,t),\quad t\geqslant 0,\quad 1\leqslant k\leqslant n,\\
z_{i+1}(u_k,t):=z_i(u_k,t\wedge\sigma_i)+\sum_{j=1}^k (z_i(u_j,t)-z_i(u_j,t\wedge\sigma_i)) \cdot \1_{A_{kj}^i},\quad t\geqslant 0,\\
1\leqslant k\leqslant n,\quad 1\leqslant i\leqslant n-1.
\end{gather*}
Here
\begin{gather*}
A_{k1}^i:=\{\sigma_i<+\infty\} \cap \{z_i(u_k,\sigma_i)-z_i(u_{k-1},\sigma_i)= \varepsilon,\ldots,z_i(u_3,\sigma_i)-z_i(u_2,\sigma_i)=\varepsilon,\\
z_i(u_2,\sigma_i)-z_i(u_1,\sigma_i)=\varepsilon\},\quad 2\leqslant k\leqslant n,\quad 1\leqslant i\leqslant n-1,
\end{gather*}
\begin{gather*}
A_{kj}^i:=\{\sigma_i<+\infty\} \cap \{z_i(u_k,\sigma_i)-z_i(u_{k-1},\sigma_i)= \varepsilon,\ldots,z_i(u_{j+1},\sigma_i)-z_i(u_j,\sigma_i)=\varepsilon,\\
z_i(u_j,\sigma_i)-z_i(u_{j-1},\sigma_i)>\varepsilon\},\quad 2\leqslant j\leqslant k-1,\quad 3\leqslant k\leqslant n,\quad 1\leqslant i\leqslant n-1,
\end{gather*}
$$
A_{11}^i:=\Omega,\quad A_{kk}^i:=\overline{\bigcup\limits_{j=1}^{k-1} A_{kj}^i},\quad 2\leqslant k\leqslant n,\quad 1\leqslant i\leqslant n-1,
$$
and for $i\in\{1,\ldots,n-1\}$ the random time $\sigma_i$ is set to be equal to
\begin{gather*}
\inf\{\left. t>\sigma_{i-1}\, \right|\, \sharp\{\left. l\in\{1,\ldots,n-1\}\, \right|\, z_i(u_{l+1},t)-z_i(u_l,t)\leqslant \varepsilon\}\geqslant\\
\geqslant \sharp\{\left. l\in\{1,\ldots,n-1\}\, \right|\, z_i(u_{l+1},\sigma_{i-1})-z_i(u_l,\sigma_{i-1})\leqslant \varepsilon\}+1\},
\end{gather*}
where the sign $\sharp$ denotes the number of elements of the corresponding set, if $\sigma_{i-1}$ is finite and to $+\infty$ otherwise, with $\sigma_0:=0$.

Note that
$$
x(u_1,t)=z_1(u_1,t)=z_2(u_1,t)=\ldots=z_n(u_1,t),\quad t\geqslant 0.
$$

We will also use the following simple generalisation of~\cite[Lemma~6.2]{Kallenberg}. (Recall that random variables $\xi$ and $\eta$ are said to be {\it equal almost surely on a (measurable) set $A\subset\Omega$} if ${\bf P} (\{\xi\neq\eta\}\cap A)=0$.)

\begin{lemma}
\label{lemma2}
Let $\xi\in L^1(\Omega,\mathcal{F}, {\bf P})$ and let $\sigma$-fields $\mathcal{G}_1,\mathcal{G}_2\subset\mathcal{F}$ be such that
$$
A\cap\mathcal{G}_1\subset A\cap\mathcal{G}_2
$$
for some $A\in\mathcal{G}_1\cap\mathcal{G}_2$. Then
$$
{\bf E} [\xi \vert \mathcal{G}_1]={\bf E} [{\bf E} [\xi \vert \mathcal{G}_2] \vert \mathcal{G}_1] \quad \text{a.~s. on $A$}.
$$	
\end{lemma}

The {\it proof} is similar to that of~\cite[Lemma~6.2]{Kallenberg}, and therefore it is omitted.

\begin{lemma}
For any $i\in\{1,\ldots,n\}$ the stochastic processes $\{z_i(u_k,t),\; t\geqslant 0\}$, $1\leqslant k\leqslant n$, are Wiener processes with respect to the initial filtration $(\mathcal{F}_t)_{t\geqslant 0}$.
\end{lemma}

\begin{proof}
We will use the principle of mathematical induction with respect to $i$.

For $i=1$ the assertion is obvious, since
$$
z_1(u_k,t)=x(u_k,t),\quad t\geqslant 0,\quad 1\leqslant k\leqslant n.
$$

Now suppose that the assertion holds true for any $i'\in\{1,\ldots,i\}$. We need to show that then it holds true for $i'=i+1$. To do this, let us fix $k\in\{2,\ldots,n\}$ and show that the stochastic process $\{z_{i+1}(u_k,t),\; t\geqslant 0\}$ satisfies the conditions of L\'{e}vy's characterising theorem.

Firstly, from its definition it can be easily seen that it has a.~s. continuous trajectories and that
$$
{\bf E} \left|z_{i+1}(u_k,t)\right|^2<+\infty,\quad t\geqslant 0.
$$

Secondly, the progressive measurability of the Wiener processes $\{z_i(u_j,t),\; t\geqslant 0\}$, $1\leqslant j\leqslant n$, implies that the sets $A_{kj}^i$, $1\leqslant j\leqslant n$, belong to the $\sigma$-field $\mathcal{F}_{\sigma_i}$ (see~\cite[Lemma~7.5]{Kallenberg}). So, from the representation
\begin{gather*}
z_{i+1}(u_k,t)=z_i(u_k,t\wedge\sigma_i)+\sum_{j=1}^k (z_i(u_j,t)-z_i(u_j,t\wedge\sigma_i)) \cdot \1_{A_{kj}^i}=\\
=z_i(u_k,t\wedge\sigma_i)+\sum_{j=1}^k (z_i(u_j,t)-z_i(u_j,t\wedge\sigma_i)) \cdot \1_{A_{kj}^i} \cdot \1\{\sigma_i\leqslant t\},\quad t\geqslant 0,
\end{gather*}
we conclude that the stochastic process $\{z_{i+1}(u_k,t),\; t\geqslant 0\}$ is $(\mathcal{F}_t)_{t\geqslant 0}$-adapted.

Thirdly, to prove that it is a martingale with respect to the filtration $(\mathcal{F}_t)_{t\geqslant 0}$ we note that for any $t\geqslant s\geqslant 0$
$$
{\bf E} \left[z_{i+1}(u_k,t)\, \left|\, \mathcal{F}_s\right.\right]={\bf E} \left[z_{i+1}(u_k,t) \cdot \1\{\sigma_i\leqslant s\}\, \left|\, \mathcal{F}_s\right.\right]+{\bf E} \left[z_{i+1}(u_k,t) \cdot \1\{\sigma_i>s\}\, \left|\, \mathcal{F}_s\right.\right].
$$
On the one hand,
\begin{gather*}
{\bf E} \left[z_{i+1}(u_k,t) \cdot \1\{\sigma_i\leqslant s\}\, \left|\, \mathcal{F}_s\right.\right]={\bf E} \left[z_i(u_k,t\wedge\sigma_i) \cdot \1\{\sigma_i\leqslant s\}\, \left|\, \mathcal{F}_s\right.\right]+\\
+\sum_{j=1}^{k} {\bf E} \left[(z_i(u_j,t)-z_i(u_j,t\wedge\sigma_i)) \cdot \1_{A_{kj}^i} \cdot \1\{\sigma_i\leqslant s\}\, \left|\, \mathcal{F}_s\right.\right]=\\
={\bf E} \left[z_i(u_k,t\wedge\sigma_i)\, \left|\, \mathcal{F}_s\right.\right] \cdot \1\{\sigma_i\leqslant s\}+\\
+\sum_{j=1}^{k} {\bf E} \left[(z_i(u_j,t)-z_i(u_j,t\wedge\sigma_i))\, \left|\, \mathcal{F}_s\right.\right] \cdot \1_{A_{kj}^i} \cdot \1\{\sigma_i\leqslant s\}=\\
=z_i(u_k,s\wedge\sigma_i) \cdot \1\{\sigma_i\leqslant s\}+\sum_{j=1}^{k} (z_i(u_j,s)-z_i(u_j,s\wedge\sigma_i)) \cdot \1_{A_{kj}^i} \cdot \1\{\sigma_i\leqslant s\}=\\
=z_{i+1}(u_k,s) \cdot \1\{\sigma_i\leqslant s\}.
\end{gather*}
On the other hand,
\begin{gather*}
{\bf E} \left[z_{i+1}(u_k,t)\, \left|\, \mathcal{F}_{\sigma_i}\right.\right]=\\
={\bf E} \left[z_i(u_k,t\wedge\sigma_i)\, \left|\, \mathcal{F}_{\sigma_i}\right.\right]+\sum_{j=1}^{k} {\bf E} \left[(z_i(u_j,t)-z_i(u_j,t\wedge\sigma_i)) \cdot \1_{A_{kj}^i}\, \left|\, \mathcal{F}_{\sigma_i}\right.\right]=\\
={\bf E} \left[z_i(u_k,t\wedge\sigma_i)\, \left|\, \mathcal{F}_{\sigma_i}\right.\right]+\sum_{j=1}^{k} {\bf E} \left[(z_i(u_j,t)-z_i(u_j,t\wedge\sigma_i))\, \left|\, \mathcal{F}_{\sigma_i}\right.\right] \cdot \1_{A_{kj}^i}=\\
=z_i(u_k,t\wedge\sigma_i),
\end{gather*}
and so, using Lemma~\ref{lemma2} in the second equality below, we obtain that
\begin{gather*}
{\bf E} \left[z_{i+1}(u_k,t) \cdot \1\{\sigma_i>s\}\, \left|\, \mathcal{F}_s\right.\right]={\bf E} \left[z_{i+1}(u_k,t)\, \left|\, \mathcal{F}_s\right.\right] \cdot \1\{\sigma_i>s\}=\\
={\bf E} \left[{\bf E} \left[z_{i+1}(u_k,t)\, \left|\, \mathcal{F}_{\sigma_i}\right.\right]\, \left|\, \mathcal{F}_s\right.\right] \cdot \1\{\sigma_i>s\}={\bf E} \left[z_i(u_k,t\wedge\sigma_i)\, \left|\, \mathcal{F}_s\right.\right] \cdot \1\{\sigma_i>s\}=\\
=z_i(u_k,s\wedge\sigma_i) \cdot \1\{\sigma_i>s\}=z_i(u_k,s) \cdot \1\{\sigma_i>s\}=z_{i+1}(u_k,s) \cdot \1\{\sigma_i>s\}.
\end{gather*}
Thus,
$$
{\bf E} \left[z_{i+1}(u_k,t)\, \left|\, \mathcal{F}_s\right.\right]=z_{i+1}(u_k,s).
$$

Finally, it remains to show that
$$
\left<z_{i+1}(u_k,\cdot)\right>_t=t,\quad t\geqslant 0.
$$
However, from the equalities
\begin{gather*}
z_{i+1}(u_k,t)=z_i(u_k,t\wedge\sigma_i)+\sum_{j=1}^k (z_i(u_j,t)-z_i(u_j,t\wedge\sigma_i)) \cdot \1_{A_{kj}^i}=\\
=z_i(u_k,t\wedge\sigma_i)+\sum_{j=1}^k z_i(u_j,t) \cdot \1_{A_{kj}^i}-\sum_{j=1}^k z_i(u_j,t\wedge\sigma_i) \cdot \1_{A_{kj}^i}
\end{gather*}
it follows that
\begin{gather*}
\left<z_{i+1}(u_k,\cdot)\right>_t=\left<z_i(u_k,\cdot)\right>_{t\wedge\sigma_i}+\sum_{j_1,j_2=1}^k \left<z_i(u_{j_1},\cdot),z_i(u_{j_2},\cdot)\right>_t \cdot \1_{A_{kj_1}^i} \cdot \1_{A_{kj_2}^i}+\\
+\sum_{j_1,j_2=1}^k \left<z_i(u_{j_1},\cdot),z_i(u_{j_2},\cdot)\right>_{t\wedge\sigma_i} \cdot \1_{A_{kj_1}^i} \cdot \1_{A_{kj_2}^i}+ 2\sum_{j=1}^k \left<z_i(u_k,\cdot),z_i(u_j,\cdot)\right>_{t\wedge\sigma_i} \cdot \1_{A_{kj}^i}-\\
-2\sum_{j=1}^k \left<z_i(u_k,\cdot),z_i(u_j,\cdot)\right>_{t\wedge\sigma_i} \cdot \1_{A_{kj}^i}-2\sum_{j_1,j_2=1}^k \left<z_i(u_{j_1},\cdot),z_i(u_{j_2},\cdot)\right>_{t\wedge\sigma_i} \cdot \1_{A_{kj_1}^i} \cdot \1_{A_{kj_2}^i}=\\
=t\wedge\sigma_i+\sum_{j=1}^k (t-t\wedge\sigma_i) \cdot \1_{A_{kj}^i}=t.
\end{gather*}
Thus, all conditions of L\'{e}vy's theorem are satisfied. The lemma is proved.
\end{proof}

\begin{lemma}
\label{lemma3}
For any $n\geqslant 2$ we have
\begin{gather*}
\sum_{k=1}^n {\bf E} \sup_{0\leqslant t\leqslant 1} \left|z_1(u_k,t)-z_2(u_k,t)\right|\leqslant \dfrac{2n^3}{3} \cdot \sqrt{\varepsilon},\\
\sum_{k=1}^n {\bf E} \sup_{0\leqslant t\leqslant 1} \left|z_i(u_k,t)-z_{i+1}(u_k,t)\right|\leqslant \dfrac{2n^4}{3} \cdot \sqrt{\varepsilon},\quad 2\leqslant i\leqslant n-1.
\end{gather*}
\end{lemma}

\begin{proof}
Let us set
$$
\widetilde{\sigma}_i:=\sigma_i\wedge 1,\quad 1\leqslant i\leqslant n-1.
$$

To prove the first estimate let us fix $k\in\{2,3,\ldots,n\}$ and note that
\begin{gather*}
{\bf E} \sup_{0\leqslant t\leqslant 1} \left|z_1(u_k,t)-z_2(u_k,t)\right|= {\bf E} \sup_{\widetilde{\sigma}_1\leqslant t\leqslant 1} \left|z_1(u_k,t)-z_2(u_k,t)\right|=\\
={\bf E} \sup_{\widetilde{\sigma}_1\leqslant t\leqslant 1} \sum_{j=1}^k \left(\left|x(u_k,t)- [x(u_j,t)+[x(u_k,\widetilde{\sigma}_1)-x(u_j,\widetilde{\sigma}_1)]]\right| \cdot \1_{A_{kj}^1}\right)=\\
=\sum_{j=1}^k {\bf E} \left(\sup_{\widetilde{\sigma}_1\leqslant t\leqslant 1} \left|x(u_k,t)-[x(u_j,t)+[x(u_k,\widetilde{\sigma}_1)-x(u_j,\widetilde{\sigma}_1)]]\right| \cdot \1_{A_{kj}^1}\right)\leqslant\\
\leqslant\sum_{j=1}^k {\bf E} \left(\sup_{0\leqslant t\leqslant 1} \left|[x(u_k,t+\widetilde{\sigma}_1)-x(u_k,\widetilde{\sigma}_1)]- [x(u_j,t+\widetilde{\sigma}_1)-x(u_j,\widetilde{\sigma}_1)]\right| \cdot \1_{A_{kj}^1}\right).
\end{gather*}
Let us estimate a separate term. To do this, fix an arbitrary $j\in\{1,\ldots,k-1\}$ (the $k$th term is obviously equal to zero) and set
\begin{gather*}
\beta_1(t):=x(u_k,t+\widetilde{\sigma}_1)-x(u_k,\widetilde{\sigma}_1),\quad t\geqslant 0,\\
\beta_2(t):=x(u_j,t+\widetilde{\sigma}_1)-x(u_j,\widetilde{\sigma}_1),\quad t\geqslant 0.
\end{gather*}
Due to the strong Markov property of the Brownian motion, the stochastic processes $\{\beta_1(t),\; t\geqslant 0\}$ and $\{\beta_2(t),\; t\geqslant 0\}$ are Wiener processes. By~\cite[Theorem~18.4]{Kallenberg} there exists (maybe on an extended probability space) a Wiener process $\{\beta(t),\; t\geqslant 0\}$ such that the representation
$$
\beta_1(t)-\beta_2(t)=\beta(\left<\beta_1-\beta_2\right>_t),\quad t\geqslant 0,\quad \text{a.~s.},
$$
takes place. Furthermore,
\begin{gather*}
\left<\beta_1-\beta_2\right>_0=0,\\
\left<\beta_1-\beta_2\right>_\cdot\in C([0;+\infty)),
\end{gather*}
and on the set $A_{kj}^1$ for all $t\geqslant 0$ we have
\begin{gather*}
\beta_1(t)-\beta_2(t)=\\
=[x(u_k,t+\widetilde{\sigma}_1)-x(u_k,\widetilde{\sigma}_1)]- [x(u_j,t+\widetilde{\sigma}_1)-x(u_j,\widetilde{\sigma}_1)]=\\
=[x(u_k,t+\widetilde{\sigma}_1)-x(u_j,t+\widetilde{\sigma}_1)]-(k-j) \cdot \varepsilon\geqslant -(k-j) \cdot \varepsilon.
\end{gather*}
It is easy to check that this implies that
$$
\left<\beta_1-\beta_2\right>_t\leqslant \tau_\beta(c_{kj}),\quad t\geqslant 0,\quad \text{a.~s. on $A_{kj}^1$},
$$
where
$$
\tau_\beta(c):=\inf\{s\geqslant 0\, \vert\, \beta(s)=c\},\quad c\in\mathbb{R},
$$
and
$$
c_{kj}:=-(k-j) \cdot \varepsilon<0.
$$
Hence
$$
\beta_1(t)-\beta_2(t)=\beta(\left<\beta_1-\beta_2\right>_t\wedge\tau_\beta(c_{kj})),\quad t\geqslant 0, \quad\text{a.~s. on $A_{kj}^1$}.
$$
In addition,
$$
0\leqslant \left<\beta_1-\beta_2\right>_t=2t-2\left<\beta_1,\beta_2\right>_t\leqslant 4t,\quad t\geqslant 0.
$$
Therefore,
\begin{gather*}
{\bf E} \left(\sup_{0\leqslant t\leqslant 1} \left|[x(u_k,t+\widetilde{\sigma}_1)-x(u_k,\widetilde{\sigma}_1)]- [x(u_j,t+\widetilde{\sigma}_1)-x(u_j,\widetilde{\sigma}_1)]\right| \cdot \1_{A_{kj}^1}\right)=\\
={\bf E} \left(\sup_{0\leqslant t\leqslant 1} \left|\beta_1(t)-\beta_2(t)\right| \cdot \1_{A_{kj}^1}\right)={\bf E} \left(\sup_{0\leqslant t\leqslant 1} \left|\beta(\left<\beta_1-\beta_2\right>_t\wedge\tau_\beta(c_{kj}))\right| \cdot \1_{A_{kj}^1}\right)\leqslant\\
\leqslant {\bf E} \left(\sup_{0\leqslant t\leqslant 4} \left|\beta(t\wedge\tau_\beta(c_{kj}))\right| \cdot \1_{A_{kj}^1}\right)\leqslant {\bf E} \sup_{0\leqslant t\leqslant 4} \left|\beta(t\wedge\tau_\beta(c_{kj}))\right|.
\end{gather*}
Applying Doob's inequality to the martingale $\{\beta(t\wedge\tau_\beta(c_{kj})),\; 0\leqslant t\leqslant 4\}$ and the second Wald identity, we obtain that
\begin{gather*}
{\bf E} \sup_{0\leqslant t\leqslant 4} \left|\beta(t\wedge\tau_\beta(c_{kj}))\right|\leqslant \sqrt{{\bf E} \sup_{0\leqslant t\leqslant 4} \left|\beta(t\wedge\tau_\beta(c_{kj}))\right|^2}\leqslant 2\sqrt{{\bf E} \left|\beta(4\wedge\tau_\beta(c_{kj}))\right|^2}=\\
=2\sqrt{{\bf E} \left(4\wedge\tau_\beta(c_{kj})\right)}\leqslant 2\sqrt{\dfrac{4\sqrt{2}}{\sqrt{\pi}} \cdot \left|c_{kj}\right|}\leqslant 4(k-j) \cdot \sqrt{\varepsilon}
\end{gather*}
(the last but one inequality follows from a simple estimate of the density of the distribution of $\tau_\beta(c_{kj})$; for details see the proof of~\cite[Lemma~5]{Dorogovtsev2010}, where a similar case was considered).

Thus, we conclude that
$$
\sum_{k=1}^n {\bf E} \sup_{0\leqslant t\leqslant 1} \left|z_1(u_k,t)-z_2(u_k,t)\right|\leqslant \sum_{k=1}^n\sum_{j=1}^k 4(k-j) \cdot \sqrt{\varepsilon}=\dfrac{2n(n^2-1)}{3} \cdot \sqrt{\varepsilon}\leqslant \dfrac{2n^3}{3} \cdot \sqrt{\varepsilon}.
$$

To prove the second estimate let us fix $i\in\{2,\ldots,n-1\}$ and $k\in\{2,\ldots,n\}$ and set
$$
B_{jl}^i:=A_{kj}^i\cap A_{kl}^{i-1},\quad 1\leqslant j\leqslant l\leqslant k.
$$
Then we note that
\begin{gather*}
{\bf E} \sup_{0\leqslant t\leqslant 1} \left|z_i(u_k,t)-z_{i+1}(u_k,t)\right|={\bf E} \sup_{\widetilde{\sigma}_i\leqslant t\leqslant 1} \left|z_i(u_k,t)-z_{i+1}(u_k,t)\right|=\\
={\bf E} \sup_{\widetilde{\sigma}_i\leqslant t\leqslant 1} \sum_{l=1}^k \sum_{j=1}^l \left(\left|[z_i(u_l,t)-z_i(u_j,t)]- [z_i(u_l,\widetilde{\sigma}_i)-z_i(u_j,\widetilde{\sigma}_i)]\right| \cdot \1_{B_{jl}^i}\right)=\\
=\sum_{l=1}^k \sum_{j=1}^l {\bf E} \left(\sup_{\widetilde{\sigma}_i\leqslant t\leqslant 1} \left|[z_i(u_l,t)-z_i(u_l,\widetilde{\sigma}_i)]- [z_i(u_j,t)-z_i(u_j,\widetilde{\sigma}_i)]\right| \cdot \1_{B_{jl}^i}\right)\leqslant\\
\leqslant\sum_{l=1}^k \sum_{j=1}^l {\bf E} \left(\sup_{0\leqslant t\leqslant 1} \left|[z_i(u_l,t+\widetilde{\sigma}_i)-z_i(u_l,\widetilde{\sigma}_i)]- [z_i(u_j,t+\widetilde{\sigma}_i)-z_i(u_j,\widetilde{\sigma}_i)]\right| \cdot \1_{B_{jl}^i}\right).
\end{gather*}
Further we proceed just as in the previous case, noting that for $1\leqslant l\leqslant k$ and $1\leqslant j\leqslant l$ on the set $B_{jl}^i$ we have
$$
z_i(u_l,t)-z_i(u_j,t)=x(u_l,t)-x(u_j,t)\geqslant 0,\quad t\geqslant 0.
$$
Thus, we conclude that
\begin{gather*}
\sum_{k=1}^n {\bf E} \sup_{0\leqslant t\leqslant 1} \left|z_i(u_k,t)-z_{i+1}(u_k,t)\right|\leqslant \sum_{k=1}^n \sum_{l=1}^k \sum_{j=1}^l 4(l-j) \cdot \sqrt{\varepsilon}=\\
=\sum_{k=1}^n \dfrac{2k(k^2-1)}{3} \cdot \sqrt{\varepsilon}\leqslant \sum_{k=1}^n \dfrac{2k^3}{3} \cdot \sqrt{\varepsilon}\leqslant \dfrac{2n^4}{3} \cdot \sqrt{\varepsilon}.
\end{gather*}

The lemma is proved.
\end{proof}

\begin{theorem}
\label{theorem3}
If $n\geqslant 2$ is such that
$$
\dfrac{1}{2}d(\Gamma)<\frac{1}{n},
$$
then
$$
W_1(\Lambda^n,\Lambda_0^n)\leqslant \dfrac{\sqrt{2}n^5}{3} \cdot \sqrt{d(\Gamma)}.
$$
\end{theorem}

\begin{proof}
Clearly, we may assume that $d(\Gamma)>0$. If we set
$$
u_k:=\dfrac{2k-1}{2n},\quad 1\leqslant k\leqslant n,
$$
then
$$
u_{k+1}-u_k=\dfrac{1}{n}>\varepsilon,\quad 1\leqslant k\leqslant n-1,
$$
where
$$
\varepsilon:=\dfrac{1}{2}d(\Gamma)>0=d(\1_{\{0\}}).
$$

Let us note that the stochastic processes $\{(z_n(u_1,t),\ldots,z_n(u_n,t)),\; t\geqslant 0\}$ and $\{(z_{0,n}(u_1,t),\ldots,z_{0,n}(u_n,t)),\; t\geqslant 0\}$ constructed according to the procedure described above (with the just defined $\varepsilon$) for the Harris flow $\{x(u,t),\; u\in\mathbb{R},\; t\geqslant 0\}$ and the Arratia flow $\{x_0(u,t),\; u\in\mathbb{R},\; t\geqslant 0\}$ respectively have the same distribution. Therefore, the distributions $\widetilde{\Lambda}^n$ and $\widetilde{\Lambda}_0^n$ of the random measures
$$
\widetilde{\lambda}^n:=\sum_{k=1}^n p_k^n\delta_{z_n(u_k,1)}
$$
and
$$
\widetilde{\lambda}_0^n:=\sum_{k=1}^n p_k^n\delta_{z_{0,n}(u_k,1)}
$$
coincide. So, by the triangle's inequality
$$
W_1(\Lambda^n,\Lambda_0^n)\leqslant W_1(\Lambda^n,\widetilde{\Lambda}^n)+ W_1(\widetilde{\Lambda}^n,\widetilde{\Lambda}_0^n)+W_1(\widetilde{\Lambda}_0^n,\Lambda_0^n)= W_1(\Lambda^n,\widetilde{\Lambda}^n)+W_1(\widetilde{\Lambda}_0^n,\Lambda_0^n).
$$
However, using Lemma~\ref{lemma3} we obtain that
\begin{gather*}
W_1(\Lambda^n,\widetilde{\Lambda}^n)=\inf_{\varkappa\in C(\Lambda^n,\widetilde{\Lambda}^n)} \iint\limits_{\mathcal{M}_1^2(\mathbb{R})} W_1(\mu',\mu'')\,\varkappa(d\mu',d\mu'')\leqslant {\bf E} W_1(\lambda^n,\widetilde{\lambda}^n)=\\
={\bf E} \inf_{\varkappa\in C(\lambda^n,\widetilde{\lambda}^n)} \iint\limits_{\mathbb{R}^2} |u-v|\,\varkappa(du,dv)\leqslant {\bf E} \sum_{k=1}^n p_k^n\left|x(u_k,1)-z_n(u_k,1)\right|\leqslant\\
\leqslant\sum_{k=1}^n {\bf E} \sup_{0\leqslant t\leqslant 1} \left|z_1(u_k,t)-z_n(u_k,t)\right| \leqslant\sum_{k=1}^n \sum_{i=1}^{n-1} {\bf E} \sup_{0\leqslant t\leqslant 1} \left|z_i(u_k,t)-z_{i+1}(u_k,t)\right|=\\
=\sum_{i=1}^{n-1} \sum_{k=1}^n {\bf E} \sup_{0\leqslant t\leqslant 1} \left|z_i(u_k,t)-z_{i+1}(u_k,t)\right|\leqslant \dfrac{2n^5}{3} \cdot \sqrt{\varepsilon}
\end{gather*}
and, similarly,
$$
W_1(\widetilde{\Lambda}_0^n,\Lambda_0^n)\leqslant \dfrac{2n^5}{3} \cdot \sqrt{\varepsilon}.
$$
This implies the desired result.
\end{proof}

\section{PROOF OF THE MAIN RESULT: THIRD STEP}

\begin{proof}[Proof of Theorem~\ref{theorem1}]
Let $n\geqslant 2$ be such that
$$
\dfrac{1}{2}d(\Gamma)<\frac{1}{n}.
$$
By the triangle's inequality we have
$$
W_1(\Lambda,\Lambda_0)\leqslant W_1(\Lambda,\Lambda^n)+W_1(\Lambda^n,\Lambda_0^n)+W_1(\Lambda_0^n,\Lambda_0).
$$
On the one hand, by Theorem~\ref{theorem2},
\begin{gather*}
W_1(\Lambda,\Lambda^n)\leqslant \dfrac{K}{\sqrt{n}},\\
W_1(\Lambda_0^n,\Lambda_0)=W_1(\Lambda_0,\Lambda_0^n)\leqslant \dfrac{K}{\sqrt{n}}.
\end{gather*}
On the other hand, by Theorem~\ref{theorem3},
$$
W_1(\Lambda^n,\Lambda_0^n)\leqslant \frac{\sqrt{2}n^5}{3} \cdot \sqrt{d(\Gamma)}.
$$
Thus, we obtain
$$
W_1(\Lambda,\Lambda_0)\leqslant 2K \cdot \left(\dfrac{1}{\sqrt{n}}+n^5 \cdot \sqrt{d(\Gamma)}\right),
$$
since
$$
2K>\frac{\sqrt{2}}{3}.
$$

The function
$$
h(y)=\dfrac{1}{\sqrt{y}}+y^5 \cdot \sqrt{d(\Gamma)},\quad y\geqslant 1,
$$
attains its minimum at the point
$$
y_0=\dfrac{1}{(10\sqrt{d(\Gamma)})^{2/11}}.
$$
Therefore, we set
$$
n_0:=\left(\left[\dfrac{1}{(10\sqrt{d(\Gamma)})^{2/11}}\right]+1\right)\in\mathbb{N}
$$
and note that the assumption $d(\Gamma)<\frac{1}{100}$ implies that $n_0\geqslant 2$ and $\frac{1}{2}d(\Gamma)<\frac{1}{n_0}$. So,
\begin{gather*}
W_1(\Lambda,\Lambda_0)\leqslant 2K \cdot \left(\dfrac{1}{\sqrt{n_0}}+n_0^5 \cdot \sqrt{d(\Gamma)}\right)\leqslant\\
\leqslant 2K \cdot \left(\sqrt{(10\sqrt{d(\Gamma)})^{2/11}}+\left(2 \cdot \dfrac{1}{(10\sqrt{d(\Gamma)})^{2/11}}\right)^5 \cdot \sqrt{d(\Gamma)}\right)=\\
=2K \cdot \left(10^{1/11} \cdot d(\Gamma)^{1/22}+\left(\dfrac{512}{25}\right)^{5/11} \cdot d(\Gamma)^{1/22}\right)=C \cdot d(\Gamma)^{1/22},
\end{gather*}
where $C:=2K \cdot \left(10^{1/11}+(512/25)^{5/11}\right)>0$. The theorem is proved.
\end{proof}

\end{document}